\documentclass[11pt,reqno]{amsart}
\usepackage{graphicx}
\usepackage{float}
\usepackage[caption = false]{subfig}
\usepackage{amsmath}
\usepackage{enumerate}
\usepackage{mathtools}
\usepackage[english]{babel}
\usepackage{amsfonts,amssymb}
\usepackage{mathrsfs}
\usepackage[utf8x]{inputenc}\usepackage[english]{babel}
\usepackage{soul}
\usepackage[all]{xy,xypic}
\usepackage{cite}
\usepackage{relsize}
\numberwithin{equation}{section}
\newtheorem{theorem}{Theorem}[section]
\newtheorem{corollary}[theorem]{Corollary}
\newtheorem{lemma}[theorem]{Lemma}

\theoremstyle{definition}
\newtheorem{definition}[theorem]{Definition}
\newtheorem{conjecture}[theorem]{Conjecture}

\theoremstyle{remark}
\newtheorem{remark}[theorem]{Remark}

\numberwithin{equation}{section}
\DeclareMathOperator{\RE}{Re}

\begin{document}
	
	\title[\tiny{A Conjecture on $H_3(1)$ For Certain Starlike Functions}]{A Conjecture on $H_3(1)$ For Certain Starlike Functions}
	
	\author{Neha Verma}
	\address{Department of Applied Mathematics, Delhi Technological University, Delhi--110042, India}
	\email{nehaverma1480@gmail.com}
	
	\author[S. Sivaprasad Kumar]{S. Sivaprasad Kumar}
	\address{Department of Applied Mathematics, Delhi Technological University, Delhi--110042, India}
	\email{spkumar@dce.ac.in}

	
\maketitle
\begin{abstract}
 We prove a conjecture concerning the third Hankel determinant, proposed in
``Anal. Math. Phys., https://doi.org/10.1007/s13324-021-00483-7",
which states that $|H_3(1)|\leq 1/9$ is sharp for the class $\mathcal{S}_{\wp}^{*}=\{zf'(z)/f(z) \prec \varphi(z):=1+ze^z\}$.
In addition, we also establish bounds for sixth and seventh coefficient, and $|H_4(1)|$ for functions in $\mathcal{S}_{\wp}^{*}$. The general bounds for two and three-fold symmetric functions related to the Ma-Minda classes $\mathcal{S}^*(\varphi)$ of starlike functions are also obtained.
\end{abstract}
\maketitle
	
\section{Introduction}
\label{intro}

For the given positive integers $n$ and $q$, the Hankel determinant $H_q(n)$ related to the function $f(z)=z+\sum_{n=2}^{\infty}a_nz^n\in \mathcal{A},$ the class of normalized analytic functions, given by
\begin{equation*}
H_{q}(n) =\begin{vmatrix}
a_n&a_{n+1}& \ldots &a_{n+q-1}\\
a_{n+1}&a_{n+2}&\ldots &a_{n+q}\\
\vdots& \vdots &\ddots &\vdots\\
a_{n+q-1}&a_{n+q}&\ldots &a_{n+2q-2}
\end{vmatrix}
\end{equation*}
where $a_1 = 1,$ was defined by Noonan and Thomas \cite{2 Noonan}. For various choices of $q$ and $n$, the growth of $H_q(n)$ was explored for many subfamilies of univalent functions.
Janteng et al. \cite{2 Janteng} discovered the sharp estimates on the second Hankel determinant,
for the classes of starlike and convex functions. Krishna et al.\cite{krishna bezilevic} calculated the best estimates of $H_2(2)$ for the class of Bazilevi$\breve{c}$ functions.

Moreover,
\begin{equation} \label{2 third Hexpression}
H_3(1):=2 a_2a_3a_4+a_3(a_5-a_3^2)-a_4^2-a_2^2a_5
\end{equation}
is the third order Hankel detrminant.
Zaprawa \cite{zaprawa} evaluated the non sharp bounds on the third Hankel determinant as $|H_3(1)|\leq 1$ and $|H_3(1)|\leq 49/540$ for the classes of starlike and convex functions, respectively. The presence of higher order coefficients in the expression of $H_3(1)$ makes it difficult to solve, and the sharpness of the solution is crucial. Overcoming the challenges, Kowaczk et al.\cite{2 kowal} established the sharp bound $|H_3(1)|\leq 4/135$ for the class of convex functions in 2018. Kwon et al.\cite{sharp starlike third hankel} provided the best known estimate for starlike functions given by $|H_3(1)|\leq 4/135$. Later on, in 2018, Lecko et al.\cite{lecko 1/2 bound} proved that $|H_3(1)|\leq 1/9$ is the sharp bound for starlike functions of order $1/2$. By choosing specific values of $\varphi(z)$ in the class $\mathcal{S}^{*}(\varphi)$ of Ma-Minda \cite{ma-minda} defined by
\begin{equation}
   \mathcal{S}^{*}(\varphi)=\bigg\{f\in \mathcal {A}:\dfrac{zf'(z)}{f(z)}\prec \varphi(z) \bigg\},\label{mindaclass}
\end{equation}
authors in \cite{shagun 3sharp} have obtained the sharp estimates for $|H_3(1)|\leq 1/36$ for the choice $\varphi(z)=\sqrt{1+z}$. Since, a proper, careful and infact precise re-arrangement of terms is highly required to obtain the best possible bound and it results in few research articles, which are available in this area for the sharp bound of $H_3(1)$, see \cite{shagun 3sharp,rath,kowal}.
The most important step in obtaining the sharp bound of $H_3(1)$ was to rewrite equation (\ref{2 third Hexpression}) in terms of Carathe\'{odory} coefficients, especially $p_1,p_2,p_3,p_4$ and $p_5$, where $p_i's$ are coefficients of the class $\mathcal{P}:=\{p(z)=1+\sum_{n=1}^{\infty}p_nz^n:\RE p(z)>0\}.$
Libera and Zlotkiewicz presented the formula for $p_2$, $p_3$ in \cite{rj, rj1}
and Kwon et al. gave the expression for $p_4$ in \cite{lemma1}, which is stated below in the form of a lemma:
\begin{lemma}
\label{2 pi}
Let $p(z)=1+\sum_{n=1}^{\infty}a_n z^n\in \mathcal {P}$. Then,
\begin{equation}
2 p_2=p_1^2+\gamma (4-p_1^2),\label{b2}
\end{equation}
\begin{equation}
4p_3=p_1^3+2p_1(4-p_1^2)\gamma -p_1(4-p_1^2) {\gamma}^2+2(4-p_1^2)(1-|\gamma|^2)\eta \label{b3}
\end{equation}
and \begin{equation}
8p_4=p_1^4+(4-p_1^2)\gamma (p_1^2({\gamma}^2-3\gamma+3)+4\gamma)-4(4-p_1^2)(1-|\gamma|^2)(p_1(\gamma-1)\eta+\bar{\gamma}{\eta}^2-(1-|\eta|^2)\rho), \label{b4}
\end{equation}
for some $\gamma$, $\eta$ and $\rho$ such that $|\gamma|\leq 1$, $|\eta|\leq 1$ and $|\rho|\leq 1.$
\end{lemma}
In 2021, Kumar and Gangania \cite{kumar-ganganiaCardioid-2021}, introduced a new class $\mathcal{S}^{*}_{\wp}$ by choosing $\varphi(z)=1+ze^z$ in (\ref{mindaclass}) defined as
\begin{equation*}
   \mathcal{S}_{\wp}^{*}=\bigg\{f\in \mathcal {A}:\dfrac{zf'(z)}{f(z)}\prec 1+ze^z\bigg\}.
\end{equation*}
They used the similar strategy and obtained the bound as, $|H_3(1)|\leq 0.150627$ for $\mathcal{S}^{*}_{\wp}$ and also proposed a conjecture which is stated as follows:
\begin{conjecture}\cite[Page no. 33]{kumar-ganganiaCardioid-2021}\label{2 conjecture}
If $f\in \mathcal{S}_{\wp}^{*},$ then the sharp bound for the third Hankel determinant is given by
\begin{equation*}
    |H_3(1)|\leq \frac{1}{9}\approx 0.1111\ldots,
\end{equation*}
with the extremal function $f(z)=z \exp\bigg(\frac{1}{3}(e^{z^3}-1)\bigg)=z+\frac{1}{3}z^4+\frac{2}{9}z^7+\cdots.$
\end{conjecture}

In this article, together with the proof of conjecture, the estimates for $a_6$ and $a_7$ in association with the fourth order Hankel determinant for the functions in the class $\mathcal{S}_{\wp}^{*}$ are obtained. For $\mathcal{S}^{*}(\varphi)$, the general third Hankel determinant for second and third fold symmetric functions are also estimated.

\section{Proof of the Conjecture \ref{2 conjecture}}
The initial coefficients for the functions in class $\mathcal {S}^{*}_{\wp}$ given in \cite{kumar-ganganiaCardioid-2021}, are as follows:
\begin{equation}
a_2=\dfrac{p_1}{2},\quad a_3=\dfrac{1}{4}\bigg(p_2+\dfrac{p_1^2}{2}\bigg),\quad a_4=\dfrac{1}{6}\bigg(p_3+\dfrac{3}{4}p_1p_2\bigg)\label{2 three coefficients}
\end{equation}

\begin{equation}
a_5=\dfrac{1}{8}\bigg(\dfrac{p_1^4}{48}+\dfrac{p_2^2}{4}+\dfrac{2p_1p_3}{3}-\dfrac{p_1^2p_2}{8}+p_4\bigg)\label{2 fifth coefficient}
\end{equation}

\begin{equation}
    a_6=\frac{1}{4}\bigg(-\dfrac{p_1^5}{240}+\dfrac{19p_1^3 p_2}{480}-\dfrac{7 p_1 p_2^2}{80} -\dfrac{p_1^2 p_3}{15} + \dfrac{p_2 p_3}{6} +\dfrac{p_1 p_4}{4} +\dfrac{2p_5}{5}\bigg)\label{2 a6}
\end{equation}
and
\begin{equation}
    a_7=\dfrac{1}{4}\bigg(\dfrac{17p_1^6}{11520}-\dfrac{37 p_1^4 p_2}{1920} +\dfrac{29 p_1^2 p_2^2}{480}-\dfrac{p_2^3}{32} +\dfrac{13p_1^3 p_3 }{360}-\dfrac{p_1 p_2 p_3 }{6}+\dfrac{p_3^2}{18}-\dfrac{p_1^2 p_4}{16}+\dfrac{p_2 p_4}{8}\label{2 a7}
+\dfrac{p_1 p_5}{5}+\dfrac{p_6}{3}\bigg).
\end{equation}

 Now, we proceed by providing a positive response to the conjecture \ref{2 conjecture} in the form of a new theorem which states:
\begin{theorem}\label{proof of conjecture}
Let $f\in \mathcal {S}^{*}_{\wp}.$ Then,
\begin{equation}
|H_3(1)|\leq \dfrac{1}{9}.\label{2 9.5}
\end{equation}
The result is sharp.
\end{theorem}
\begin{proof}  As the Carathe\'{odory} class is rotationally invariant, $p_1$ lies inside the interval [0,2]. On substituting the expressions of $a_i$'s from equations (\ref{2 three coefficients}) and (\ref{2 fifth coefficient}) in equation \eqref{2 third Hexpression} assuming $p_1:=p$, we get

$$H_3(1)=\dfrac{1}{9216}\bigg(3p^6-12p^4p_2+96p^3p_3-192pp_2p_3-144p^2p_2^2+144p^2p_4+72p_2^3 -256p_3^2+288p^2p_4\bigg).$$
On simplification using (\ref{b2})-(\ref{b4}), we obtain
	
$$H_3(1)=\dfrac{1}{9216}\bigg(\nu_1(p,\gamma)+\nu_2(p,\gamma)\eta+\nu_3(p,\gamma){\eta}^2+\phi(p,\gamma,\eta)\rho\bigg)$$
where $\gamma,\eta,\rho\in \mathbb {D},$
\begin{align*}	\nu_1(p,\gamma):&=-4p^6-25p^2{\gamma}^2(4-p^2)^2-5p^2{\gamma}^3(4-p^2)^2+2p^2{\gamma}^4(4-p^2)^2
+5p^4{\gamma}(4-p^2)\\
&\quad+36{\gamma}^3(4-p^2)^2-16p^4{\gamma}^2(4-p^2),\\
	\nu_2(p,\gamma):&=8(1-|\gamma|^2)(4-p^2)(4p^3-(4-p^2)(p\gamma+10p\gamma^2)),\\
	\nu_3(p,\gamma):&=8(1-|\gamma|^2)(4-p^2)^2(-8-|\gamma|^2)\\
	\phi(p,\gamma,\eta):&=72(1-|\gamma|^2)(4-p^2)^2(1-|\eta|^2)\gamma.
	\end{align*}
On taking $x=|\gamma|,y=|\eta|$ and using the fact that $|\rho|\leq 1,$ we have
\begin{align*}
|H_3(1)|\leq \dfrac{1}{9216}\bigg(|\nu_1(p,\gamma)|+|\nu_2(p,\gamma)|y+|\nu_3(p,\gamma)|y^2+|\phi(p,\gamma,\eta)|\bigg)\leq B(p,x,y),
\end{align*}
where
\begin{equation}
B(p,x,y)=\dfrac{1}{9216}\bigg(b_1(p,x)+b_2(p,x)y+b_3(p,x)y^2+b_4(p,x)(1-y^2)\bigg)\label{2 new}
\end{equation}
with
\begin{align*}
b_1(p,x):&=4p^6+25p^2x^2(4-p^2)^2+5p^2x^3(4-p^2)^2+2p^2x^4(4-p^2)^2+5p^4x(4-p^2)\\
	&\quad+36x^3(4-p^2)^2+16p^4x^2(4-p^2),\\
b_2(p,x):&=8(1-x^2)(4-p^2)(4p^3+(4-p^2)(px+10px^2)),\\	b_3(p,x):&=8(1-x^2)(4-p^2)^2(8+x^2),\\
b_4(p,x):&=72(1-x^2)(4-p^2)^2x.
\end{align*}
We must maximise $B(p,x,y)$ in the cuboid $V:[0,2]\times [0,1]\times [0,1]$. So, we use the maximum
values in the interiors of the six faces, the twelve edges, and the interior of $V$.
\begin{enumerate}
\item To begin with, all the interior points of $V$ are taken into consideration. On partially differentiating equation (\ref{2 new}) with respect to $y$ to determine its points of maxima in the interior of $V$ by considering $(p,x,y)\in (0,2)\times (0,1)\times (0,1)$. We get
\begin{align*}
\dfrac{\partial B}{\partial y}&=\dfrac{1}{1152}(4 - p^2) (1 - x^2) \bigg(4p^3 + (4- p^2) (px(1+10 x)-18 x y+2y(8 + x^2))\bigg).
\end{align*}
Now $\dfrac{\partial B}{\partial y}=0$ gives
\begin{equation*}
y=y_0:=\dfrac{4px(1+10x)-p^3(10x^2+x-4)}{2(x-1)(x-8)(p^2-4)}.
\end{equation*}
We note that $y_0\in (0,1)$ assuring the existence of critical points and it
 is possible when
\begin{equation}
	   4px(1+10x)-p^3(10x^2+x-4)+64-72x+2x^2(4-p^2)<2p^2(8-9x).\label{2 h1}
\end{equation}

Now we will apply the hit and trial method to find solutions that satisfy the inequalities (\ref{2 h1})
for the existence of critical point. If we assume that $p$ tends to 0 and 2, there is no $x\in (0,1)$ that can satisfy the equation (\ref{2 h1}). Similarly, if $x$ tends to 0 and 1, there is no $p \in (0,2)$ satisfying equation (\ref{2 h1}). 
As a result, the function $B$ does not have a critical point in $(0,2)\times(0,1)\times(0,1)$.
\item Now, the interiors of all the faces of $V$ are taken into consideration.\\
On the face $p=0, B(p,x,y)$ becomes
\begin{equation}
	        d_1(x,y):=B(0,x,y)=\dfrac{9x^3+2(1-x^2)((8+x^2)y^2+9x(1-y^2))}{144}\label{2 9.4}
\end{equation}
with $x,y\in (0,1)$. We notice that $d_1$ has no critical points in $(0,1)\times(0,1)$ as
\begin{equation}
    \dfrac{\partial d_1}{\partial y}=\dfrac{-(1-x)^2(x+1)(x-8)y}{36}\neq 0,\quad x,y\in (0,1).
\end{equation}
On the face $p=2, B(p,x,y)$ becomes
\begin{equation}
    B(2,x,y):=\dfrac{1}{36},\quad x,y\in (0,1).\label{2 9.3}
\end{equation}
On the face $x=0, B(p,x,y)$ becomes
\begin{equation}
    d_2(p,y):=B(p,0,y)=\dfrac{(p^3+16y-4p^2y)^2}{2304}\label{2 9.1}
\end{equation}
with $y\in (0,1)$ and $p\in (0,2).$ To identify the points of maxima, we solve $\partial d_2/\partial p$ and $\partial d_2/\partial y$. On solving $\partial d_2/\partial y=0,$ we obtain
\begin{equation}
    y=-\dfrac{p^3}{4(4-p^2)}(=:y_1).\label{2 52}
\end{equation}
Based on simple calculations, we can conclude that such $y_1$ does not belong to $(0,1)$.
As a result, there is no critical point in $(0,2)\times(0,1)$ for $d_2$.\\
On the face $x=1, B(p,x,y)$ becomes
\begin{equation}
    d_3(p,y):=B(p,1,y)=\dfrac{576+224p^2-136p^4+15p^6}{9216}, \quad p\in (0,2).\label{2 9.2}
\end{equation}
During the calculations, $p_0:=p\approx0.991758$ comes out to be the critical point when $\partial d_3/\partial p=0$.  Moreover, using basic mathematics, $d_3$ reaches its maximum value of approximately $ 0.0736789$ at $p_0$.\\
On the face $y=0, B(p,x,y)$ becomes
\begin{align*}
    B(p,x,0)&=\dfrac{1}{9216}\bigg(4p^4x(23-34x-19x^2-4x^3)+p^6(4-5x+9x^2+5x^3+2x^4)\\
    &\quad\quad\quad\quad+576x(2-x^2)+16p^2x(-36+25x+23x^2+2x^3)\bigg)\\
    &\quad\quad\quad\quad=:d_4(p,x).
\end{align*}
On computing,
\begin{align*}
    \dfrac{\partial d_4}{\partial x}&=\dfrac{1}{9216}\bigg(-4p^4(-23+34x+19x^2+4x^3)+p^6(-5+18x+15x^2+8x^3)\\
    &\quad\quad\quad\quad-4p^4x(34+38x+12x^2)+16p^2(-36+25x+23x^2+2x^3)\\
    &\quad\quad\quad\quad-1152x^2-576(x^2-2)+16p^2x(25+46x+6x^2) \bigg)
\end{align*}
and \begin{align*}
    \dfrac{\partial d_4}{\partial p}&=\dfrac{p}{4608}\bigg(16x(-36+25x+23x^2+2x^3)-8p^2x(-23+34x+19x^2+4x^3)\\
    &\quad\quad\quad\quad+3p^4(4-5x+9x^2+5x^3+2x^4)\bigg).
\end{align*}
According to a numerical computation, the critical points obtained by solving the system of equations $\partial d_4/\partial x=0$ and $\partial d_4/\partial p=0$ in $(0,2)\times (0,1)$ implies
\begin{equation*}
    B(0.00115734,0.816497,0)\leq 0.0680414.
\end{equation*}
On the face $y=1, B(p,x,y)$ reduces to
\begin{align*}
    B(p,x,1)&=\dfrac{1}{9216}\bigg(4p^4(16+5x-48x^2-x^3-6x^4)-8p^5(4-x-14x^2+x^3+10x^4)\\
    &\quad\quad\quad\quad+p^6(4-5x+9x^2+5x^3+2x^4)+16p^2(-32+53x^2-13x^3+6x^4)\\
    &\quad\quad\quad\quad+128px(1+10x-x^2-10x^3)+64(16-14x^2+9x^3-2x^4)\\
    &\quad\quad\quad\quad+64p^3(2-x-12x^2+x^3+10x^4)\bigg)=:d_5(p,x).
\end{align*}
We observe that the system of equations $\partial d_5/\partial x=0$ and $\partial d_5/\partial p=0$ has no solution in $(0,2)\times (0,1).$\\
\item Now, we determine the maximum values attained on the edges of the cuboid $V$ by $B(p,x,y)$. We get $B(p,0,0)=c_1(p):=4p^6/9216$ from equation (\ref{2 9.1}). It is apparent that $c_1'(p)=0$ for $p=\beta_0:=0$ and $p=\beta_1:=2$ as points of minima and maxima in the interval $[0,2]$, respectively. Maximum value of $c_1(p)$ is $\approx 0.0277778$. Hence,
\begin{equation*}
    B(p,0,0)\leq 0.0277778.
\end{equation*}
With $y=1,$ we get $B(p,0,1)=c_2(p):=(16-4p^2+p^3)2/2304$ from the equation (\ref{2 9.1}).
Since $c_2'(p)<0$
in the range $[0,2]$, $p=0$ acts as its point of maxima. Thus
\begin{equation*}
    B(p,0,1)\leq \dfrac{1}{9}, \quad p\in [0,2].
    \end{equation*}
Based on computations, $B(0,0,y)$ in equation (\ref{2 9.1}) reaches its maximum value at $y=1$. This leads to
\begin{equation*}
    B(0,0,y)\leq \dfrac{1}{9}, \quad y\in [0,1].
\end{equation*}
Since, the equation (\ref{2 9.2}) is free from $x$, we have $B(p,1,1)=B(p,1,0)=c_3(p):=(15p^6-136p^4+224p^2+576)/9216.$ Now, $c_3'(p)=448p-544p^3+90p^5=0$ when $p=\beta_2:=0$ and $p=\beta_3:=0.991758$ in the interval $[0,2]$ where $\beta_2$ and $\beta_3$ are the minimum and maximum points, respectively. Hence
\begin{equation*}
    B(p,1,1)=B(p,1,0)\leq 0.0736789,\quad p\in [0,2].
\end{equation*}
We get $B(0,1,y)=1/16$ when $p=0$ is substituted in equation (\ref{2 9.2}). The equation (\ref{2 9.3}) do not involve any variables namely $p$, $x$ and $y$. Thus, on the edges $p=2, x=1; p=2, x=0; p=2, y=0;$ and $p=2, y=1,$ the maximum value of $B(p,x,y)$ is determined by
\begin{equation*}
    B(2,1,y)=B(2,0,y)=B(2,x,0)=\dfrac{1}{36},\quad x,y\in [0,1].
\end{equation*}
From equation (\ref{2 9.1}), we obtain $B(0,0,y)=y^2/36.$ A simple calculation shows that
\begin{equation*}
    B(0,0,y)\leq \dfrac{1}{36},\quad y\in [0,1].
    \end{equation*}
We get $B(0,x,1)=c_4:=(8-7x^2-x^4)/72$ using equation (\ref{2 9.4}).
We observe that $c_4$ is a decreasing function in the range $[0,1]$, and so reaches its maximum value at $x=0,$ according to a simple calculation. Hence

\begin{equation*}
     B(0,x,1)\leq \dfrac{1}{9},\quad x\in [0,1].
\end{equation*}
$B(0,x,0)=c_5(x):=x(1-x^2)/8$ is obtained using equation (\ref{2 9.4}) once again. We get $c_5'(x)=0$ after further calculations for $x=x_0:=1/\sqrt{3}$. Also, $c_5(x)$ is an increasing function in $[0,x_0)$ and decreasing in $(x_0,1].$ As a result, maximum value is attained at $x_0.$ Thus

\begin{equation*}
     B(0,x,0)\leq 0.0481125,\quad x\in [0,1].
 \end{equation*}
\end{enumerate}
Thus, the inequality (\ref{2 9.5}) holds for all the cases. The sharpness of the result is governed by the function $f:\mathbb{D}\rightarrow \mathbb{C}$ defined as
\begin{equation}
f(z)=z\exp\bigg(\dfrac{1}{3}(e^{z^3}-1)\bigg)=z+\dfrac{z^4}{3}+\dfrac{2z^7}{9}+\cdots,\label{2 extremal}
\end{equation}
with $f(0)=0$ and $f'(0)=1$. For the values of $a_2=a_3=a_5=0$ and $a_4=1/3$ the function specified in equation (\ref{2 extremal}) serves as an extremal function for the bounds of $H_3(1)$.

\end{proof}
To prove our next result, we must first recall the following lemma:
\begin{lemma}\label{2 pomi lemma}
Let $p=1+\sum_{n=1}^{\infty}p_nz^n\in \mathcal{P}.$ Then
\begin{equation}
    |p_n|\leq 2, \quad n\geq 1,\label{2 caratheodory1}
\end{equation}
\begin{equation}
    |p_{n+k}-\mu p_n p_k|\leq \begin{cases}
    2, & 0\leq \mu\leq 1;\\
    2|2\mu-1|,& elsewhere,
    \end{cases}\label{2 caratheodory2}
\end{equation}

and \begin{equation}
    |p_1^3-\mu p_3|\leq
    \begin{cases}2|\mu-4|,& \mu\leq 4/3;\\ \\
    2\mu\sqrt{\dfrac{\mu}{\mu-1}},& 4/3<\mu.
    \end{cases}\label{2 caratheodory3}
\end{equation}
\end{lemma}

\begin{lemma}\label{2 a6a7bound}
Let $f\in \mathcal{S}_{\wp}^{*}.$ Then $|a_6|\leq 47/60\approx 0.7833$ and $|a_7|\leq 503/480\approx 1.0479.$
\end{lemma}
\begin{proof}
From equation (\ref{2 a6}), we have
\begin{align*}
 |a_6|&=\bigg|\frac{1}{4}\bigg(-\dfrac{p_1^5}{240}+\dfrac{19p_1^3 p_2}{480}-\dfrac{7 p_1 p_2^2}{80} -\dfrac{p_1^2 p_3}{15} + \dfrac{p_2 p_3}{6} +\dfrac{p_1 p_4}{4} +\dfrac{2p_5}{5}\bigg)\bigg|\\
 &=\bigg|\dfrac{1}{10}\bigg(p_5+\dfrac{5}{8}p_1p_4\bigg)-\dfrac{p_1^2}{60}\bigg(p_3-\dfrac{19}{32}p_1p_2\bigg)+\dfrac{p_2}{24}\bigg(p_3-\dfrac{21}{40}p_1p_2\bigg)-\dfrac{2}{1920}p_1^5\bigg|.
\end{align*}
By using equation(\ref{2 caratheodory1}), (\ref{2 caratheodory2}) and triangle inequality, we get
\begin{equation*}
    |a_6|\leq \dfrac{47}{60}.
\end{equation*}
Similarly, from equation (\ref{2 a7}), we have
\begin{align*}
   46080a_7&=17 p_1^6 - 222 p_1^4 p_2 + 696 p_1^2 p_2^2 - 360 p_2^3 + 416 p_1^3 p_3 -1920 p_1 p_2 p_3 + 640 p_3^2 - 720 p_1^2 p_4\\
   &\quad + 1440 p_2 p_4+ 2304 p_1 p_5 +3840 p_6
 \end{align*}
 or
 \begin{align*}
46080|a_7|&\leq3840\bigg|p_6+\dfrac{3p_2p_4}{8}\bigg|+416|p_1|^3\bigg|p_3-\dfrac{111p_1p_2}{208}\bigg|+360|p_2|^2\bigg|-p_2+\dfrac{29p_1^2}{15}\bigg|\\
&\quad+2304|p_1|\bigg|p_5-\dfrac{5p_2p_3}{6}\bigg|
 +|17p_1^6-720p_1^2p_4|+640|p_3|^2.
\end{align*}
On using equation (\ref{2 caratheodory1}) and (\ref{2 caratheodory2}), we get
\begin{equation}
    3840\bigg|p_6+\dfrac{3p_2p_4}{8}\bigg|\leq 13440,\quad 416|p_1|^3\bigg|p_3-\dfrac{111}{208}p_1p_2\bigg|\leq 6656,\label{2 7(1)}
\end{equation}
\begin{equation}
  360|p_2|^2\bigg|p_2-\dfrac{29p_1^2}{15}\bigg|\leq 7008,\quad 2304|p_1|\bigg|p_5-\dfrac{5p_2p_3}{6}\bigg|\leq 9216  \label{2 7(2)}
\end{equation}
and \begin{equation}
   |17p_1^6-720p1^2p_4|+640|p_3|^2 \leq 11968.\label{2 7(3)}
\end{equation}
From equations (\ref{2 7(1)}), (\ref{2 7(2)}) and (\ref{2 7(3)}), we get the desired result.
\end{proof}

\section{ Fourth Hankel Determinant and third hankel for n-fold symmetric functions}
\subsection{Fourth Hankel estimation for the class $\mathcal{S}^*_{\wp}$}
For $q=4$ and $n=1$, the expression of the fourth Hankel determinant can be written as
\begin{equation}
H_4(1)=a_7H_3(1)-a_6Q_1+a_5Q_2-a_4Q_3,
\end{equation}
where
\begin{equation}
    Q_1=a_3(a_2a_5-a_3a_4)-a_4(a_5-a_2a_4)+a_6(a_3-a_2^2),\label{2 Q1}
\end{equation}
\begin{equation}
    Q_2=a_3(a_3a_5-a_4^2)-a_5(a_5-a_2a_4)+a_6(a_4-a_2a_3),\label{2 Q2}
\end{equation}
and
\begin{equation}
   Q_3=a_4(a_3a_5-a_4^2)-a_5(a_2a_5-a_3a_4)+a_6(a_4-a_2a_3). \label{2 Q3}
\end{equation}
To compute the fourth Hankel determinant for the function in class $\mathcal{S}_{\wp}^{*},$ we substitute the values of $a_i$ from equations (\ref{2 three coefficients})-(\ref{2 a7})
 in equation (\ref{2 Q1}).
Upon simplification, we have
\begin{align*}
    92160Q_1&=27p_1^7-408p_1^5p_2+p_1^2p_2(660p_1^2-1224p_2)+p_1^2(336p_2p_3-1152p_5)\\
    &\quad +p_2(2304p_5-480p_2p_3)+p_4(1440p_1p_2-1920p_3)+392p_1^4p_3.
\end{align*}
By applying Lemma \ref{2 pomi lemma}, we arrive at
\begin{align*}
   92160 |Q_1|&\leq  |p_1|^5|27p_1^2-408p_2|+|p_1|^2|p_2||660p_1^2-1224p_2|+|p_1|^2|336p_2p_3-1152p_5|\\
    &\quad +|p_2||2304p_5-480p_2p_3|+|p_4||1440p_1p_2-1920p_3|+392|p_1|^4|p_3|\\
    &\leq 84352.
\end{align*}
So, we obtain
\begin{equation}
    |Q_1|\leq \dfrac{659}{720}\approx 0.915278.\label{2 Q1value}
\end{equation}
In similar way, we have
\begin{align*}
 737280Q_2&=73p_1^8+1440p_1^5p_3+p_1^2p_4(14400p_2-1920p_1^2)+p_1p_2p_3(4096p_1^2-8448p_2)\\
 &\quad +p_1^4p_2(-1692p_2-336p_1^2)
 +p_5(12288p_3-4608p_1^3)-11520p_4^2+720p_2^4\\
 &\quad -4608p_1^2p_3^2.
\end{align*}
Using Lemma \ref{2 pomi lemma}, we get
\begin{align*}
 737280|Q_2|&\leq |p_1|^5|73p_1^3+1440p_3|+|p_1|^2|p_4||14400p_2-1920p_1^2|+|p_1||p_2||p_3||4096p_1^2-8448p_2|\\
 &\quad+|p_1|^4|p_2||-1692p_2-336p_1^2|
 +|p_5||12288p_3-4608p_1^3|+11520|p_4|^2+720|p_2|^4\\
 &\quad+4608|p_1|^2|p_3|^2\\
 &\leq 759296+98304\sqrt{\frac{2}{5}}.
\end{align*}
So, we get
\begin{equation}
    |Q_2|\leq \dfrac{759296+98304\sqrt{\frac{2}{5}}}{737280}\approx 1.11419.\label{2 Q2value}
\end{equation}
Again, by rearrangement of terms, we have,
\begin{align*}
 4423680Q_3&=57p_1^9+288p_1^6p_3+p_1^3p_4(25920p_2-5760p_1^2)+p_1^4p_2(16128p_3-144p_1^3)\\
 &\quad-p_3^2(20480p_3+6144p_1^3)-p_1^3p_2^2(9216p_2+540p_1^2)+p_1p_5(36864p_3-6912p_1^3)\\
 &\quad-26496p_1^2p_2^2p_3+12528p_1p_2^4+46080p_2p_3p_4-34560p_1p_4^2-27648p_2^2p_5.
 \end{align*}
 From Lemma \ref{2 pomi lemma}, we get,
\begin{align*}
 4423680|Q_3|&\leq |p_1|^6|57p_1^3+288p_3|+|p_1|^3|p_4||25920p_2-5760p_1^2|+|p_1|^4|p_2||16128p_3-144p_1^3|\\
 &\quad+|p_3|^2|20480p_3+6144p_1^3|+|p_1|^3|p_2|^2|9216p_2+540p_1^2|+|p_1||p_5||36864p_3-6912p_1^3|\\
 &\quad+26496|p_1|^2|p_2|^2|p_3|+12528|p_1||p_2|^4+46080|p_2||p_3||p_4|+34560|p_1||p_4|^2\\
 &\quad+27648|p_2|^2|p_5|\\
 &\leq4029952+1376256\sqrt{\frac{21}{37}}+\frac{1179648}{\sqrt{13}}.
 \end{align*}
So, we get
\begin{equation}
    |Q_3|\leq \frac{4029952+1376256\sqrt{\frac{21}{37}}+\frac{1179648}{\sqrt{13}}}{4423680}\approx 1.21934.\label{2 Q3value}
\end{equation}
Based on the above computations, we make the following statement on fourth Hankel determinant:
\begin{theorem}
Let $f\in \mathcal{S}_{\wp}^{*}.$ Then $|H_4(1)|\leq 2.54589.$
\end{theorem}

\subsection{Third Hankel determinant for $2\& 3$ fold symmetric functions for $\mathcal{S}^*(\varphi)$ }
In the recent times, it has been observed that finding the sharp estimates of third Hankel determinant for general Ma-Minda class is not feasible till now. But for some classes, sharp estimates have been obtained, for instance, see \cite{shagun 3sharp,rath,kowal} and now including Theorem \ref{proof of conjecture} as well which motivated us to settle the Conjecture \ref{2 conjecture}. Further looking at the difficulty of the general class, we restrict ourselves to answer the problem for the $n$-fold symmetric functions.

\begin{definition}\cite{goodman vol1}
A function $f\in \mathcal{A}$ is called $n$-fold symmetric if $f(e^{2\pi i/n}z)=e^{2\pi i/n}f(z)$ which holds for all $z\in \mathbb{D}$ and $n$ is a natural number. We denote the set of $n$-fold symmetric functions by $\mathcal{A}^{(n)}.$
\end{definition}

Let $f\in \mathcal{A}^{(n)},$ then $f$ has power series expansion
 $$f(z)=a_1z+a_{n+1}z^{n+1}+a_{2n+1}z^{2n+1}+\cdots.$$
Therefore, for $f\in \mathcal{A}^{(3)}$ and $f\in \mathcal{A}^{(2)}$ respectively, we have
\begin{equation}
    H_3(1)=-a_4^2 \quad \text{and} \quad H_3(1)=a_3(a_5-a_3^2). \label{2 nfold}
\end{equation}

Now we conclude this paper with the following result:
\begin{theorem}
\label{2 thm nfold}
Let $f\in \mathcal{S}^{*}{(\varphi)}.$ Then
\begin{enumerate}
\item $\widehat{f}\in \mathcal{S}^{*(3)}(\varphi)$ implies that $|H_3(1)|\leq |B_1|^2/9.$ \label{3-fold}
\item $\widehat{f}\in \mathcal{S}^{*(2)}(\varphi)$ implies that
\begin{equation*}
|H_3(1)|\leq \dfrac{1}{4}|B_1|\times\begin{cases}
\frac{1}{6}(B_2-\frac{9}{8}B_1^2+B_1^2),& \frac{9}{4}B_1^2\leq 2(B_2+B_1^2-B_1);\\\\
\frac{1}{6}B_1, & 2(B_2+B_1^2-B_1)\leq \frac{9}{4}B_1^2\leq 2(B_2+B_1^2+B_1);\\\\
\frac{1}{6}(-B_2+\frac{9}{8}B_1^2-B_1^2), & 2(B_2+B_1^2+B_1)\leq \frac{9}{4}B_1^2.
\end{cases}\label{2-fold}
\end{equation*}
\end{enumerate}
The estimate in \eqref{3-fold} is sharp.
\end{theorem}
\begin{proof}
Since $f(z)=z+a_2 z^2+a_3 z^3+\cdots\in\mathcal{S}^{*}{(\varphi)}.$ Let
 $$\varphi(z)=1+B_1z+B_2z^2+B_3z^3+\cdots$$
and
$$p(z)=\frac{zf'(z)}{f(z)}=1+b_1z+b_2z^2+\cdot.$$ This equation shows that
\begin{equation}
    (n-1)a_n=\sum_{k=1}^{n-1}b_ka_{n-k}\quad n>1.\label{shellybi}
\end{equation}
 Since $\varphi$ is univalent and $p\prec \varphi$, then the function $$p_1(z)=\frac{1+\varphi^{-1}(p(z))}{1-\varphi^{-1}(p(z))}=1+c_1z+c_2z^2+c_3z^3+\cdots,$$
belongs to the class $\mathcal{P}.$ Or equivalently,
 $$p(z)=\varphi\bigg(\frac{p_1(z)-1}{p_1(z)+1}\bigg).$$
Using the last equation, the coefficient $b_i$ can be expressed in terms of $c_i$ and $B_i$ ($i\in \mathbb{N})$. We have
\begin{equation*}
     b_1=\dfrac{1}{2}B_1c_1,\quad b_2=\dfrac{1}{4}((B_2-B_1)c_1^2+2B_1c_2)
\end{equation*}
and \begin{equation*}
     b_3=\dfrac{1}{8}((B_1-2B_2+B_3)c_1^3+4(B_2-B_1)c_1c_2+4B_1c_3).
 \end{equation*}
 Hence, by using the expressions for $b_k$ in equation (\ref{shellybi}), we obtain
 \begin{equation}
     a_2=b_1=\frac{1}{2}B_1c_1,\quad\text{and}\quad a_3=\frac{1}{8}((B_1^2-B_1+B_2)c_1^2+2B_1c_2).\label{2 BCi}
 \end{equation}
\begin{enumerate}
\item Since $f(z)=z+a_2 z^2+a_3 z^3+\cdots \in \mathcal{S}^{*}(\varphi)$ if and only if
    $$\widehat{f}(z)=(f(z^3))^{1/3}=z+\beta_4z^4+\cdots \in \mathcal{S}^{*(3)}(\varphi).$$ We have $\beta_4=b_1/3.$ Hence, for $f\in \mathcal{S}^{*(3)}(\varphi),$ from equation (\ref{2 nfold}), we obtain
    \begin{align}
        H_3(1)&=|\beta_4|^2=\dfrac{1}{9}|b_1|^2=\dfrac{1}{36}|B_1 c_1|^2\\
        &\leq \dfrac{1}{9}|B_1|^2.\label{2-h313}
    \end{align}
Also, the result is sharp for $f_0(z)=z\exp \int_{0}^{z}\frac{\varphi(t)-1}{t}dt$ and its rotations.
\item  Since $f(z)=z+a_2 z^2+a_3 z^3+\cdots \in \mathcal{S}^{*}(\varphi)$ if and only if
$$\widehat{f}(z)=(f(z^2))^{1/2}=z+\alpha_3z^3+\alpha_5 z^5\cdots \in \mathcal{S}^{*(2)}(\varphi).$$ Upon comparing the coefficients in the following:
\begin{equation*}
        z^2+a_2z^4+a_3z^6+\cdots=(z++\alpha_3z^3+\alpha_5 z^5\cdots)^2,
\end{equation*}
we obtain
\begin{equation}
        \alpha_3=\dfrac{1}{2}a_2 \quad \text{and} \quad \alpha_5=\dfrac{1}{2}a_3-\dfrac{1}{8}a_2^2.
\end{equation}
If $f\in \mathcal{S}^{*(2)}(\varphi)$, then from equation (\ref{2 nfold}), we have
\begin{equation}
        H_3(1)=\alpha_3(\alpha_5-\alpha_3^2)\label{2-h312}
\end{equation}
Using equation (\ref{2 BCi}),
\begin{align*}
        | H_3(1)|&=\bigg|\dfrac{1}{2}a_2\bigg(\dfrac{4a_3-3a_2^2}{8}\bigg)\bigg|
        =\bigg|\dfrac{a_2}{4}\bigg(a_3-\dfrac{3}{4}a_2^2\bigg)\bigg|\\
        &\leq \dfrac{1}{4}|a_2|\bigg|a_3-\dfrac{3}{4}a_2^2\bigg|.
    \end{align*}
Now, 
equation (\ref{2 BCi})
and Fekete-Szeg\"{o} bounds~\cite[Theorem 3]{ma-minda} for $\mu=3/4$, we get the desired result.
\end{enumerate}
\end{proof}
\begin{corollary}
Let $f\in \mathcal{S}^{*}{(\varphi)}$ and $H_3(1)$ is given by equation (\ref{2-h313}). Then
\begin{enumerate}
    \item $\widehat{f}\in \mathcal{S}^{*(3)}(1+ze^z)$ implies that $|H_3(1)|\leq 1/9.$
    \item $\widehat{f}\in \mathcal{S}^{*(3)}((1+z)/(1-z))$ implies that $|H_3(1)|\leq 4/9.$
    \end{enumerate}
The sharpness of the bounds follows from \cite{kumar-ganganiaCardioid-2021} and \cite{zaprawa} respectively.
\end{corollary}
\begin{corollary}
\label{particular-2fold}
Let $f\in \mathcal{S}^{*}{(\varphi)}$ and $H_3(1)$ is given by equation (\ref{2-h312}). Then
\begin{enumerate}
    \item $\widehat{f}\in \mathcal{S}^{*(2)}(1+ze^z)$ implies that $|H_3(1)|\leq 1/24.$
    \item $\widehat{f}\in \mathcal{S}^{*(2)}((1+z)/(1-z))$ implies that $|H_3(1)|\leq 1/6.$
\end{enumerate}
\end{corollary}
\begin{remark}
We observe that the bounds obtained in Corollary \ref{particular-2fold} are close to the sharp values and are still open for their sharpness.
\end{remark}

	%
\subsection*{Acknowledgment}
Neha is thankful to the Department of Applied Mathematics, Delhi Technological University, New Delhi-110042 for providing Research Fellowship.

\end{document}